\documentclass[10pt]{amsart}
\usepackage{latexsym,amssymb, amsmath, amsthm,xcolor,tikz-cd,hyperref,mathtools, geometry}
\usepackage[all, cmtip]{xy} 
\usepackage{enumerate}
\hypersetup{
    colorlinks=true,
	allcolors=black,
	linkbordercolor={1 1 1}
}

\usepackage{color}

\DeclareFontFamily{U}{wncy}{}
\DeclareFontShape{U}{wncy}{m}{n}{<->wncyr10}{}
\DeclareSymbolFont{mcy}{U}{wncy}{m}{n}
\DeclareMathSymbol{\Sha}{\mathord}{mcy}{"58}

\begin{document}
\newcommand{\Z}{\mathbb{Z}}
\newcommand{\Q}{\mathbb{Q}}
\newcommand{\F}{\mathbb{F}}
\newcommand{\kbar}{\overline{k}}
\newcommand{\pro}{\mathrm{Pro}}
\newcommand{\colim}{\mathrm{colim}}
\newcommand{\Fp}{\mathbb{F}_p}
\newcommand{\Hom}{\mathrm{Hom}}
\newcommand{\Spec}{\mathrm{Spec}}
\newcommand{\et}{\acute{e}t}
\newcommand{\Et}{\acute{E}t}
\newcommand{\eT}{\textup{\'et}}
\newcommand{\Etr}{\acute{E}t^{\natural}_{/k}}
\newcommand{\xto}{\xrightarrow}
\newcommand{\Gal}{\mathrm{Gal}}
\newcommand{\Aut}{\mathrm{Aut}}
\newcommand{\nth}{n^{\text{th}}}
\newcommand{\A}{\mathbb{A}}
\newcommand{\Br}{\operatorname{Br}}
\newcommand{\Jac}{\operatorname{Jac}}
\newcommand{\Pic}{\operatorname{Pic}}
\newcommand{\Xbar}{\overline{X}}
\newcommand{\Ybar}{\overline{Y}}
\newcommand{\Res}{\operatorname{Res}}
\newcommand{\Sel}{\operatorname{Sel}}
\newcommand{\ksep}{k^s}
\newcommand{\res}{\operatorname{res}}
\newcommand{\NS}{\operatorname{NS}}

\newtheorem{thm}{Theorem}[section]
\newtheorem{cor}[thm]{Corollary}
\newtheorem{lemma}[thm]{Lemma}
\newtheorem{propn}[thm]{Proposition}
\newtheorem{prop}[thm]{Property}
\theoremstyle{definition}
\newtheorem{example}[thm]{Example}
\newtheorem{defn}[thm]{Definition}
\newtheorem{rem}[thm]{Remark}
\newtheorem{con}[thm]{Construction}
\newtheorem{qst}[thm]{Question}
\newtheorem{notation}[thm]{Notation}
\newtheorem{conj}[thm]{Conjecture}
\title{Galois invariants of finite abelian descent and Brauer sets}
\author{Brendan Creutz, Jesse Pajwani and Jos\'e Felipe Voloch}
\address{School of Mathematics and Statistics, University of Canterbury, Private
 Bag 4800, Christchurch 8140, New Zealand}
\email{brendan.creutz,jesse.pajwani,felipe.voloch@canterbury.ac.nz}


\maketitle

\begin{abstract}
For a variety over a global field, one can consider
subsets of the set of adelic points of the variety cut out 
by finite abelian descent or Brauer-Manin obstructions. Given a
Galois extension of the ground field one can consider similar sets over the extension
and take Galois invariants. In this paper, we study under which circumstances
the Galois invariants recover the obstruction sets over the ground field.
As an application of our results, we study finite abelian descent and
Brauer-Manin obstructions for isotrivial curves over function fields and
extend results obtained by the first and last authors for constant curves to
the isotrivial case.
\end{abstract}

\section{Introduction}

Let $X$ be a variety over a global field $k$ (i.e., a separated scheme of finite type $X \to \Spec(k)$) and assume that $X$ is geometrically integral. It is well known that the set $X(k)$ of $k$-rational points of $X$ may fail to be dense in the set $X(\A_k)$ of adelic points of $X$, as there can be cohomological obstructions mediated by the Brauer group $\Br(X)$ and/or by the finite abelian descent obstruction \cite{PoonenBook,SkorobogatovBook} (there are other obstructions, e.g. descent by all linear algebraic groups, but these are not the focus of this paper). It is clear that the assignment $L \mapsto X(L)$ defines a sheaf of sets on $\Spec(k)_\eT$ meaning that, $X(L)^{\Gal(L/k)} = X(k)$ for any finite Galois extension $L/k$. We investigate whether the sets cut out by these cohomological obstructions also define sheaves in this sense.

For a torsor $f: Y \to X$ under a finite abelian group scheme $G$ over $k$, we recall the definition of the set $X(\A_k)^f \subset X(\A_k)$ of adelic points which survive $f$ \cite[Definition 5.2]{Stoll}. The torsor $f : Y \to X$ represents a class in the fppf cohomology group $H^1(X,G)$. For each place $v$ of $k$, this defines a map $X(k_v) \ni x_v \mapsto x_v^*f \in H^1(k_v,G)$ from the set of points over the completion $k_v$ of $k$ to the fppf cohomology group $H^1(k_v,G)$. Then $X(\A_k)^f$ consists of those adelic points $x = (x_v)$ such that the family of local cohomology classes $(x_v^*f) \in \prod_v H^1(k_v,G)$ lies in the diagonal image of the global cohomology group $H^1(k,G)$. 

For a finite separable extension $L/k$, let $X(\A_L)^f$ to denote the set of $L$-adelic points on $X$ which survive the base changed torsor $f_L : Y_L \to X_L$. There is a canonical inclusion $X(\A_k) \subset X(\A_L)$. When $L/k$ is Galois, the natural action of $\Gal(L/k)$ on $X(\A_L)$ gives a bijection $X(\A_k) = X(\A_L)^{\Gal(L/k)}$ and an inclusion $X(\A_k)^f  \subset (X(\A_L)^f)^{\Gal(L/k)}$ (See \cite[Proposition 5.16]{Stoll}). Our first main result is that $L \mapsto X(\A_L)^f$ defines an \'etale sheaf if and only if the only separable point of $G$ is the identity. We note that any nontrivial group scheme over a number field contains nontrivial separable points. However if $k$ has characteristic $p > 0$, then there are nontrivial group schemes such as $G = \mu_p$ which have no nontrivial separable points.

\begin{thm}\label{thm:MainThm2}
	Let $X$ be a geometrically integral variety over a global field $k$ admitting a geometrically irreducible torsor $f: Y \to X$ under a finite abelian group scheme $G$ over $k$. Then the following are equivalent.
	\begin{enumerate}[(a)]
		\item\label{cond:1} For every finite separable extension $K/k$ and finite Galois extension $L/K$ we have $$X(\A_K)^f = (X(\A_L)^f)^{\Gal(L/K)}\,.$$
		\item\label{cond:2} For $\ksep$ the separable closure of $k$, $G(\ksep) = 0$.
	\end{enumerate}
\end{thm}

When $k$ is a number field with real places, it is fairly easy to construct examples of varieties $X$ with $X(\A_k)^f \ne (X(\A_L)^f)^{\Gal(L/k)}$ by modifying an adelic point in $X(\A_k)^f$ at a real place. In Theorem~\ref{thm:etale} we prove a stronger result showing that the failure of $L \mapsto X(\A_L)^f$ to define an \'etale sheaf when $G(\ksep) \ne 0$ as asserted in Theorem~\ref{thm:MainThm2} is not solely accounted for by issues at the archimedean places. 

When $X$ is proper, the space $X(\A_k)_\bullet$ of connected components of $X(\A_k)$ is naturally identified with the product of the spaces $X(k_v)_\bullet$ of connected components of the $X(k_v)$. The evaluation maps $X(k_v) \to H^1(k_v,G)$ featuring in the definition of $X(\A_k)^f$ are locally constant and so must factor through $X(k_v)_\bullet$. Thus, when $X$ is proper one is naturally led, as in \cite{Stoll}, to consider the set $X(\A_k)_\bullet^f$ of connected components which contain a point of $X(\A_k)^f$. 

For a finite Galois extension $L/k$, the inclusion $X(\A_k) \subset X(\A_L)$ induces maps $X(\A_k)_\bullet \to X(\A_L)_\bullet$ and $X(\A_k)^f_\bullet \to X(\A_L)^f_\bullet$ which are injective, except possibly when there is a real place $v$ of $k$ lying below only complex places of $L$. In the exceptional case $\prod_{w \mid v}X(L_w)_\bullet = \prod_{w \mid v}X(\mathbb{C})_\bullet$ is a singleton, and the induced action of $\Gal(L/k)$ is trivial.
This gives an action of $\Gal(L/k)$ on $X(\A_L)_\bullet$ such that $X(\A_L)_\bullet^{\Gal(L/k)}$ is equal to the image of $X(\A_k)_\bullet$ in $X(\A_L)_\bullet$. With this notation, Condition~\eqref{cond:1} in Theorem~\ref{thm:MainThm2} implies the following:
\begin{enumerate}[(a')]
\item\label{cond:1b} For every finite separable extension $K/k$ and finite Galois extension $L/K$, the image of $X(\A_K)_\bullet^f$ in $X(\A_L)_\bullet$ is equal to $(X(\A_L)_\bullet^f)^{\Gal(L/K)}$.
\end{enumerate}
Our results show that~\eqref{cond:1} and (a') are in fact equivalent when $X$ is proper (See Remark~\ref{rem:a'}). We will say that $L \mapsto X(\A_L)_\bullet^f$ defines an \'etale sheaf if (a') holds. Note, however, that if $k$ contains real places this may fail to define a sheaf of sets on $\Spec(k)_{\textup{\'et}}$. It is a presheaf that satisfies the axiom of glueing, but the axiom of locality can fail because the maps $X(\A_K)^f_\bullet \to X(\A_L)^f_\bullet$ may not be injective.

For a smooth, projective and geometrically integral variety $X$ over $k$ define the finite abelian descent set $X(\A_k)_\bullet^{\textup{f-ab}} := \bigcap_{f : Y \to X} X(\A_k)_\bullet^f$, where the intersection is taken over all (not necessarily geometrically irreducible) torsors under finite abelian group schemes. For clarity, we emphasise that the sets $X(\A_L)_\bullet^{\textup{f-ab}}$ for finite extensions $L/k$ are formed by taking the intersection over torsors defined over $L$, not just those arising as base change from $k$. This finite abelian descent set is closely related to the Brauer-Manin obstruction. By \cite[Theorem 6.1.1]{SkorobogatovBook} there are containments 
\[
	X(k) \subset \overline{X(k)} \subset X(\A_k)_\bullet^{\Br} \subset X(\A_k)_\bullet^{\textup{f-ab}}\,,
\]
where $X(\A_k)_\bullet^{\Br}$ denotes the set of (connected components containing) adelic points orthogonal to the Brauer group of $X$, and $\overline{X(k)}$ denotes the topological closure of (the image of) $X(k)$ in $X(\A_k)_\bullet$. If, moreover, $X \subset A$ is a subvariety of an abelian variety one also has
\begin{equation}\tag{$\dagger$}\label{diagram1}
	\xymatrix{
	\overline{X(k)} \ar@{}[r]|-*[@]{\subset}\ar@{}[d]|-*[@]{\subset} & X(\A_k)_\bullet^{\Br} \ar@{}[r]|-*[@]{\subset}\ar@{}[d]|-*[@]{\subset} & X(\A_k)_\bullet^{\textup{f-ab}} \ar@{}[d]|-*[@]{\subset}\\
	X(\A_k)_\bullet \cap \overline{A(k)} \ar@{}[r]|-*[@]{\subset} & X(\A_k)_\bullet \cap A(\A_k)_\bullet^{\Br} \ar@{}[r]|-*[@]{\subset} & X(\A_k)_\bullet \cap A(\A_k)_\bullet^{\textup{f-ab}}\,.
	}
\end{equation}
For the following conjecture we refer to {\cite[p. 133]{SkorobogatovBook}, \cite[Section 8]{Stoll}, \cite[Question 7.4]{PoonenHeuristic} and \cite[Conjecture C]{PoonenVoloch}}.

\begin{conj}\label{conj1}
	For $X$ a smooth closed subvariety of an abelian variety $A$ over a global field $k$ all of the containments in~\eqref{diagram1} are equalities.	
\end{conj}
This conjecture has been proven when $A(k)$ and $\Sha(A/k)$ are finite \cite{Scharaschkin} and for most nonisotrivial coset-free subvarieties of abelian varieties over global function fields \cite{PoonenVoloch}, including all nonisotrivial curves of genus at least $2$ over global function fields \cite{CVNonIsotriv}.

It is not difficult to show that the bottom left set of~\eqref{diagram1} defines an \'etale sheaf (see Lemma~\ref{lem:1}). The following theorem states that the bottom right set of~\eqref{diagram1} also defines an \'etale sheaf, as predicted by Conjecture~\ref{conj1}.

\begin{thm}\label{thm:Main2}
	Suppose $X \subset A$ is a smooth closed subvariety of an abelian variety over a global field $k$. Then for any finite Galois extension $L/k$ the image of $X(\A_k)_\bullet \cap A(\A_k)_\bullet^{\textup{f-ab}}$ in $X(\A_L)_\bullet$ is equal to $\left(X(\A_L)_\bullet \cap A(\A_L)_\bullet^{\textup{f-ab}} \right)^{\Gal(L/k)}\,.$
\end{thm}

\begin{cor}\label{cor:Main2}
	Let $X \subset A$ be a smooth closed subvariety of its Albanese variety embedded by an Albanese map (i.e., a morphism $\phi:X \to A$ determined by a $0$-cycle of $Z$ degree $1$ on $X$ and defined by $\phi(P) = [P-Z] \in \operatorname{Alb}^0_X = A$). Assume that the N\'eron-Severi group of $X$ is torsion free. Then for any finite Galois extension $L/k$ the image of $X(\A_k)_\bullet^{\textup{f-ab}}$ in $X(\A_L)_\bullet$ is equal to $\left(X(\A_L)_\bullet^{\textup{f-ab}} \right)^{\Gal(L/k)}\,.$
\end{cor}

\begin{rem}
In the setting of Theorem~\ref{thm:Main2} and its corollary the sets $A(\A_k)^{\textup{f-ab}}$ and $X(\A_k)^{\textup{f-ab}}$ are equal to the analogous sets obtained by intersecting only the obstruction sets coming from {\it geometrically irreducible} torsors under finite abelian group schemes. This follows from \cite[Lemma 5.7(2)]{Stoll}, since the multiplication by $n$ maps $n : A \to A$ (and their pullback to $X$) define a cofinal family of coverings. Thus, for $X$ as in the corollary, we see that the failure of $L \mapsto X(\A_L)_\bullet^f$ to define an \'etale sheaf as asserted in Theorem~\ref{thm:MainThm2} goes away in the limit.
\end{rem}

\begin{rem}
	Theorem~\ref{thm:Main2} and its corollary do not hold if one considers the space of adelic points instead of its space of connected components. For example, consider an elliptic curve $X/\Q$ with trivial Mordell-Weil group and trivial Tate-Shafarevich group. Then $X(\A_\Q)^{\textup{f-ab}}$ is the set of adelic points $(x_p) \in \prod X(\Q_p)$ such that $x_p$ lies in the connected component of the identity in $X(\Q_p)$ for all $p \le \infty$ by \cite[Corollary 6.2]{Stoll}. Any $(x_p) \in X(\A_\Q)$ such that $x_p$ is the identity of $X(k_v)$ for the nonarchimedean primes will lie in $X(\A_L)^{\textup{f-ab}}$, for any totally imaginary extension $L/\Q$. If $X(\mathbb{R})$ has two connected components, then $X(\A_\Q)^{\textup{f-ab}}$ will be a proper subset of $(X(\A_L)^\textup{f-ab})^{\Gal(L/\Q)}$.
\end{rem}

In Section~\ref{sec:BM} we verify that the other terms in~\eqref{diagram1} define \'etale sheaves in a number of cases, including the case when $X$ is a curve embedded in its Jacobian (See Theorem~\ref{thm:Curves}). We also show that, in general, none of the sets in the top row of~\eqref{diagram1} define \'etale sheaves if we consider varieties that do not embed into an abelian variety (See Propositions~\ref{BrNotSheaf} and~\ref{cor:enriques}). These examples complement various recent results studying the behaviour of obstructions under base extension \cite{CL,LiangHu,RV,Wu}.

Part of our motivation for studying these questions comes from the desire to extend the results of \cite{CVBM} concerning the Brauer-Manin obstruction for constant curves over global function fields to the case of isotrivial curves. Recall that a variety $X$ over the function field $k = \F(D)$ of a smooth projective curve $D$ over a finite field $\F$ is called constant if there is a variety $X_0/\F$ such that $X \simeq X_0 \times_\F k$. A variety $X/k$ is called isotrivial if there is a finite extension $L/k$ such that $X \times_k L$ is constant. We note that, unlike the nonisotrivial case, it is not immediate that $\overline{X(k)}$ is equal to the subset of $\overline{X(L)}$ fixed by Galois, as there are isotrivial curves of every genus with $X(k) \ne \overline{X(k)}$.

 As an application of our results we generalise \cite[Theorems 1.1 and 1.2]{CVBM} to the case of isotrivial curves. See Theorems~\ref{thm:Frobdescent} and~\ref{thm:MW} below. We also deduce the following.

\begin{thm}
	Let $X$ be a smooth projective and isotrivial curve over a global function field $k$. Suppose $X$ becomes constant after base change to $L/k$. Then $\overline{X(k)} = X(\A_k)^{\Br}$ if and only if $\overline{X(L)} = X(\A_L)^{\Br}$. Moreover, Conjecture~\ref{conj1} holds for all isotrivial curves over global function fields if it holds for all constant curves over global function fields.
\end{thm}

Note that we have $X(\A_k) = X(\A_k)_\bullet$ in the function field case. 
Some instances where the equality $\overline{X(L)} = X(\A_L)^{\Br}$ is known to hold for a constant curve over $L$ are given in \cite[Theorem 2.14]{CVV} and \cite[Theorem 1.5]{CVBM}.

\subsection*{Acknowledgements} The authors were supported by the Marsden Fund administered by the Royal Society of New Zealand. They thank the referees for their careful reading and helpful comments.

\section{Descent sets for torsors under finite abelian group schemes}

In this section we will prove Theorem~\ref{thm:MainThm2}.

\subsection{The nontrivial \'etale subgroup scheme case}\label{etale}
For a finite abelian group scheme $G$ over $k$, the base change $G_{\ksep} = G \times_{\Spec(k)}\Spec(\ksep)$ contains a unique maximal \'etale subgroup scheme $G_e$. It is determined by the property that $G_e(\ksep) = G(\ksep)$ (see \cite[Proposition 1.31]{MilneAG}). If $k$ is a number field, then $G_e = G$. If $k$ is a global function field and the characteristic of $k$ divides the order of $G$, then $G_e$ may be a proper subgroup scheme of $G_{\ksep}$. In any case $G_e$ will be defined over a finite separable extension of $k$. In the results that follow we are free to pass to a finite separable extension, so we can assume $G_e$ is defined over $k$. 

The map $G \to G/G_e$ is \'etale, so any separable point of the quotient will lift to a separable point of $G$. It follows that $(G/G_e)(\ksep) = 0$. The connected-\'etale sequence~\cite[Prop. 13.4]{MilneAG} for $G/G_e$ is an exact sequence $0 \to G_c \to G/G_e \to H_e \to 0$, where $G_c$ connected (it is the connected component of the identity in $G/G_e$) and $H_e$ \'etale. 

\begin{rem} The \'etale quotient $H_e$ may be nontrivial. For example, suppose $G$ is a nontrivial extension $1 \to \mu_p \to G \to \Z/p\Z \to 0$ where $p$ is the characteristic of $k$ (such group schemes arise as the kernel of multiplication by $p$ on an elliptic curve with $j$-invariant in $k \setminus k^p$). Then, in the notation above, we have $G_e = 0$, $G/G_e = G$, $G_c = \mu_p$ and $H_e = \Z/p\Z$.
\end{rem}

\begin{thm}\label{thm:etale}
	Let $G$ be a finite abelian group scheme over a global field $k$ such that $G(\ksep) \ne 0$. Let $X$ be a geometrically integral variety over $k$ and let $f:Y \to X$ be a geometrically irreducible torsor under $G$. Then there exists a finite separable extension $K/k$ and a finite Galois extension $L/K$ such that
	$$X(\A_K)^f \ne (X(\A_L)^f)^{\Gal(L/K)}\,.$$
	Moreover, for any $N \ge 1$, the extensions $L/K/k$ can be chosen so that $(X(\A_L)^f)^{\Gal(L/K)} \subset X(\A_K)$ contains a point which differs from all points in $X(\A_K)^f$ at at least $N$ nonarchimedean places. In particular, the assignment $L \mapsto X(\A_L)_\bullet^f$ does not define an \'etale sheaf.
\end{thm}

Before proving this theorem we will prove three lemmas. The first is sufficient to prove the theorem in the case that $G$ is \'etale. We note that for \'etale $G$, the fppf cohomology group $H^1(k,G)$ agrees with~\'etale and Galois cohomology.

\begin{lemma}\label{lem:notasheaf}
Let $X$ be geometrically integral variety over a global field $k$ and let $f:Y \to X$ be a geometrically irreducible torsor under a finite \'etale abelian group scheme $G/k$ with $G \ne 0$. Suppose $Y(\A_k) \ne \emptyset$. For any $N \ge 1$, there is a finite Galois extension $L/k$ and a point in $X(\A_k) \cap X(\mathbb{A}_L)^f$ which differs from every point in $X(\A_k)^f$ at at least $N$ nonarchimedean places.
\end{lemma}

\begin{proof}

Call a class $\xi \in H^1(k,G)$ finitely supported if the image of $\xi$ under the restriction map $H^1(k,G) \to  H^1(k_v,G)$ is trivial for all but finitely many places $v$ of $k$. Let $K/k$ be the splitting field of $G$, i.e., the minimal Galois extension $K/k$ such that $G(K) = G(\kbar)$. By \cite[I.9.3]{MilneADT} any finitely supported class lies in the image of the inflation map $H^1(K/k,G) \to H^1(k,G)$. It follows that $H^1(k,G)$ contains only finitely many finitely supported classes. As we explain in the next paragraph $H^1(k,G)$ is infinite, so there are infinitely many classes that are not finitely supported.

We now prove that $H^1(k,G)$ is infinite. For any place $v$ which splits completely in $K$, we have $H^1(k_v,G) = \Hom(\Gal(k_v),G(\ksep))$, which is nontrivial by local class field theory and the assumption that $G \ne 0$. Moreover, for any finite set $T$ of such places $v$, the product of the restriction maps $H^1(k,G) \to \prod_{v\in T}H^1(k_v,G)$ is surjective by \cite[Theorem 9.2.3(vii)]{CoN}. As there are infinitely many places of $k$ which split completely in $K$, it follows that $H^1(k,G)$ is infinite.

We now claim that there are infinitely many places $v$ of $k$ such that the map $f:Y(k_v) \to X(k_v)$ is not surjective. To see this, let $\xi \in H^1(k,G)$ be a class which is not finitely supported and let $f^\xi : Y^\xi \to X$ be the corresponding twist. Then $Y^\xi$ is geometrically irreducible (being a twist of $Y$, which is assumed to be geometrically irreducible) and $Y$ is geometrically reduced because $X$ is geometrically reduced and $G$ is \'etale. Then $Y^\xi(k_v)$ is nonempty for all but finitely many places $v$, as a consequence of the Lang-Weil estimates, Hensel's Lemma and the fact that the smooth locus of $Y^\xi$ is nonempty (See \cite[Theorem 7.7.2]{PoonenBook}). A point $x \in f^\xi(Y^{\xi}(k_v)) \subset X(k_v)$ lies in $f(Y(k_v))$ if and only if $\operatorname{res}_v(\xi) = 0$. Since $\xi$ is not finitely supported, there are infinitely many places where $f^\xi(Y^\xi(k_v))$ is a nonempty subset of $X(k_v)$ disjoint from $f(Y(k_v))$.

Let $M$ be the maximum number of places at which a finitely supported class in $H^1(k,G)$ is nonzero and let $N \ge 1$. Let $w_1,\dots,w_{M+N}$ be $M+N$ nonarchimedean places of $k$ such that $f:Y(k_{w_i}) \to X(k_{w_i})$ is not surjective. Let $(x_v) \in Y(\mathbb{A}_k)$ and choose $y_{w_i} \in X(k_{w_i}) \setminus f(Y(k_{w_i}))$ for $i = 1,\dots,M+N$. Define $(z_v) \in X(\mathbb{A}_k)$ by setting $z_v = f(x_v) \in X(k_v)$ for $v \notin \{w_1,\dots,w_{M+N}\}$ and setting $z_{w_i} = y_{w_i}$ for $i = 1,\dots,M+N$.

For all places $v \not\in \{w_1,\dots,w_{M+N}\}$ we have that $z_v^*f = f(x_v)^*f = 0 \in H^1(k_v, G)$. It follows that $(z_v)^*f$ is nonzero precisely at the $M+N$ places $w_1,\dots,w_{M+N}$. By definition of the integer $M$ we have that, for any global class $\xi \in H^1(k,G)$, the difference $z_v^*f - \xi_v$ is nonzero at at least $N$ nonarchimedean places. It follows that $(z_v) \in X(\A_k)$ differs from every element of $X(\A_k)^f$ at at least $N$ nonarchimedean places.

Since $G$ is \'etale we have that for each $i = 1,\dots,M+N$, there is a finite Galois extension $L^{w_i}/k_{w_i}$ that kills $z_{w_i}^*f$. Moreover, we can find a global Galois extension $L/k$ such that, for all $i = 0,\dots,M+N$ and all primes $u | w_i$ we have that the extension $L_{u}/k_{w_i}$ kills $z_{w_i}^*f$.
Indeed, it is enough to construct $L/k$ such that for all $i = 1,\dots,M+N$ and all $u | w_i$ we have $L_{u} \supset L^{w_i}$ and the main result of \cite{Krull} gives such an extension.

Consider $(z_v) \in X(\A_k) = X(\A_L)^{\Gal(L/k)}$, and write $z_w \in X(L_w)$ for the component of $(z_v)$. If $w$ is a place of $L$ lying over a place $v$ of $k$, we have $z_w^*(f)$ is the image of $z_v^*(f)$ under the map $H^1(k_v, G) \to H^1(L_w, G)$. By construction of $L$, we have $z_w^*(f)=0$ for all places $w$ of $L$ so $(z_w) \in X(\A_L)^f \cap X(\A_k)$ as required.
\end{proof}

\begin{lemma}\label{lem:Z}
	Let $X$ be a variety over $k$ and let $g : Z \to X$ be a torsor under a finite abelian group scheme $G$ over $k$. Assume that $X$ is smooth over $k$ and $Z(k) \ne \emptyset$. Then the reduced subscheme $Z_\textup{red} \subset Z$ contains a nonempty open subscheme that is smooth over $k$.
\end{lemma}

\begin{proof}
	By the connected-\'etale sequence \cite[Prop. 13.4]{MilneAG} the torsor $g$ factors as a composition of torsors $g_c : Z \to Z_e$ and $g_e:Z_e \to X$ under a connected group scheme and an \'etale group scheme respectively. Since $g_e$ is \'etale, we have that $Z_e$ is smooth. It therefore suffices to prove the lemma assuming that $G$ is connected. 

	Let $z \in Z(k)$. Then the fiber of $g$ above the smooth closed point $g(z) \in X(k)$ is isomorphic to the geometrically connected scheme $G$. So the local ring of $Z$ at $z$ is isomorphic to the local ring of $G$ at the identity and the spectrum of the local ring of $Z_\textup{red}$ at $z$ is the reduced subscheme of $G$. Since $G$ is finite and connected, $G_\textup{red}$ is the spectrum of a field containing $k$ as a subfield. Since $G_\textup{red}(k) = G(k) \ne \emptyset$, we must have $G_\textup{red} = \Spec(k)$, which is smooth. So $Z_\textup{red}$ is smooth at $z$. Since smooth is an open condition, this proves the lemma.	
\end{proof}

\begin{lemma}\label{lem:redet}
	Suppose $Y$ is a variety over $k$ and $f : Y \to Z$ is a torsor under a finite abelian \'etale group scheme $G$ over $k$. The pullback of $f$ along $Z_\textup{red} \to Z$ yields a torsor $f' : Y_\textup{red} \to Z_\textup{red}$ under $G$ and the canonical map $Z_{\textup{red}} \to Z$ gives bijections $Z_{\textup{red}}(\A_k) = Z(\A_k)$ and $Z_{\textup{red}}(\A_k)^{f'} = Z(\A_k)^f$.
\end{lemma}

\begin{proof}
The pullback of $f$ along $Z_\textup{red} \to Z$ is a $G$-torsor $\tilde{Y} \to Z_\textup{red}$. Since $G$ is \'etale and $Z_\textup{red}$ is reduced, we have that $\tilde{Y}$ is reduced. Since $\tilde{Y}$ is reduced, the morphism $\tilde{Y} \to Y$ must factor through $Y_\textup{red}$ by \cite[Lemma 26.12.7]{StacksProject}. Moreover, the restriction of $f$ to $Y_\textup{red}$ gives a compatible map to $Z_\textup{red}$, so we must have $\tilde{Y} = Y_\textup{red}$ by the universal property of pullbacks.

Since $\Spec(k_v)$ is reduced (as a scheme over $\Spec(k)$) any morphism $\Spec(k_v) \to Z$ factors through $Z_\textup{red}$, so the obvious inclusion $Z_\textup{red}(\A_k) \to Z(\A_k)$ is a bijection. Similarly we have bijections $Y_\textup{red}^\xi(\A_k) = Y^\xi(\A_k)$ for any twist of $f : Y \to Z$. Since $G$ is \'etale, the formation of $f'$ from $f$ commutes with twisting, so 
we also get a bijection $Z_{\textup{red}}(\A_k)^{f'} = Z(\A_k)^f$.
\end{proof}

\begin{proof}[Proof of Theorem~\ref{thm:etale}]
Since $X$ is geometrically integral, its smooth locus is nonempty and open. Replacing $X$ with its smooth locus we can assume it is smooth over $k$. We are free to replace $k$ by a finite separable extension, so we may assume that $X(k) \ne \emptyset$. Then there exists a twist $f':Y' \to X$ of $Y \to X$ with $Y'(k) \ne \emptyset$. Since every twist of $f'$ is also a twist of $f$ (and vice versa) we have that $X(\A_k)^{f'} = X(\A_k)^f$ (cf. \cite[Lemma 5.3(5)]{Stoll}). The same is true over any finite separable extension $L/k$, so replacing $Y$ by this twist if necessary we may assume $Y(k) \ne \emptyset$.

Passing to a finite separable extension we can also assume $G(k) = G(\ksep)$. Then the maximal \'etale subgroup scheme $G_e \subset G$ is defined over $k$ by \cite[Prop 1.31]{MilneAG}. As described at the beginning of this section we have an exact sequence $0 \to G_e \to G \to G/G_e \to 0$. Consequently, $f$ factors as a composition of torsors $f_e : Y \to Z$ and $g : Z \to X$ under $G_e$ and $G/G_e$, respectively. 

By Lemma~\ref{lem:Z} there is a smooth open subscheme $U \subset Z_\textup{red}$. By Lemma~\ref{lem:redet} the pullback of $f_e$ along $U \to Z_\textup{red} \to Z$ is a $G_e$-torsor $\tilde{f}_e : V \to U$, where $V$ is a nonempty open subscheme of $Y_\textup{red}$. As $Y$ is assumed to be geometrically irreducible, we have that $V$, $Z_\textup{red}$ and $U$ are also geometrically irreducible. Since smooth implies geometrically reduced, we have that $U$ is geometrically integral. Thus, the torsor $\tilde{f}_e : V \to U$ satisfies the conditions of Lemma~\ref{lem:notasheaf}. This combined with Lemma~\ref{lem:redet} gives a point $z \in Z(\A_k)$ which lies in $Z(\A_L)^{f_e} \setminus Z(\A_k)^{f_e}$.

Let $x = g(z)$. By construction $x$ lies in $g\left(Z(\A_L)^{f_e} \cap Z(\A_k)\right) \subset X(\A_L)^f \cap X(\A_k)$. To complete the proof it is enough to show that $x \notin X(\A_k)^f$. By way of contradiction, suppose $x \in X(\A_k)^f \subset X(\A_k)^g$. Then $x$ lifts to a twist $f' = g' \circ f_e'$ of $f = g \circ f$ by a class in $H^1(k,G)$ and $x$ lifts to the twist $g'$ of $g$ by the image $\xi$ of this class in $H^1(k,G/G_e)$. Since $x$ also lifts under $g$ we have $$\xi \in \Sha^1(K,G/G_e) := \bigcap_v \ker\left( \res_v : H^1(k,G/G_e) \to H^1(k_v,G/G_e)\right)\,.$$ Since $(G/G_e)(\ksep) = 0$, we have that $\Sha^1(k,G/G_e) = 0$ by \cite[Main Theorem]{GA-T}. So $g = g'$ and we see that $x \in g(Z(\A_k)^{f_e'})$. Noting that $Z(\A_k)^{f_e} = Z(\A_k)^{f_e'}$ we see that $x = g(z')$ for some $z' \in Z(\A_k)^{f_e}$. For each place $v$ of $k$, we have $g(z_v) = g(z'_v)$. Since $G/G_e(\ksep) = 0$ we must have $z_v = z'_v$ at each place, since $k_v$ points are separable \cite[Lemma 3.1]{PoonenVoloch}. So in fact $z = z'$. But this is impossible as $z \not\in Z(\A_k)^{f_e}$. 
\end{proof}

\subsection{The case $G(\ksep) = 0$}

It remains to prove Theorem~\ref{thm:MainThm2} in the case that $G$ does not contain a nontrivial \'etale subscheme. This follows immediately from the next result.

\begin{thm}\label{thm:Gksep0}
	Let $G$ be a finite abelian group scheme over a global field $k$ such that $G(\ksep)=0$. Let $X$ be a geometrically integral variety and let $f : Y \to X$ be a torsor under $G$. Then for any finite Galois extension $L/k$, $\left(X(\A_L)^f\right)^{\Gal(L/k)} = X(\A_k)^f$.
\end{thm}

\begin{lemma}\label{lem:autaia}
	Let $G$ be a finite abelian group scheme over a field $K$ and let $L/K$ be a Galois extension such that $G(L) = 0$. Then the restriction map $H^1(K,G) \to H^1(L,G)^{\Gal(L/K)}$ is an isomorphism.
\end{lemma}

\begin{proof}
	The inflation-restriction sequence in fppf cohomology (see \cite[p. 422]{Shatz}) gives an exact sequence
	\[
		H^1(L/K,G(L)) \stackrel{\inf}\to H^1(K,G) \stackrel{\res}\to H^1(L,G)^{\Gal(L/K)} \to H^2(L/K,G(L))\,.
	\]
	The outer two terms are trivial because $G(L) = 0$. Thus, the restriction map is an isomorphism.
\end{proof}

\begin{proof}[Proof of Theorem~\ref{thm:Gksep0}]
	Noting that $X(\A_L)^{\Gal(L/k)} = X(\A_k)$, it will be enough to show that we have $X(\mathbb{A}_k)^{f} = X(\mathbb{A}_L)^f \cap X(\mathbb{A}_k)$. The inclusion $\subseteq$ is clear, so we show the reverse inclusion. We have a commutative diagram
	\begin{equation}\label{diagram:autaia}
		\xymatrix{
		H^1(k, G) \ar@{^{(}->}[r] \ar@{^{(}->}[d] & \prod_v H^1(k_v, G) \ar@{^{(}->}[d] \\
 		H^1(L,G) \ar@{^{(}->}[r] & \prod_v\prod_{w \mid v} H^1(L_w, G)
 		}
	\end{equation}
	where the injectivity of the vertical maps come from Lemma~\ref{lem:autaia} and the injectivity of the horizontal maps is~\cite[Main Theorem]{GA-T}.

	Let $x \in X(\mathbb{A}_L)^f \cap X(\mathbb{A}_k)$ and consider the image $x^*f$ of $x$ under the map $X(\A_k) \to \prod_v H^1(k_v,G)$ induced by $f$. The image of $x^*f$ in $\prod_v \prod_{w \mid v} H^1(L_w, G)$ is the image of a unique $\xi \in H^1(L,G)$ under the bottom horizontal map of~\eqref{diagram:autaia} because $x \in X(\mathbb{A}_L)^f$. For any $v$ there is a natural action of $\Gal(L/k)$ on $\prod_{w\mid v}H^1(L_w,G)$ which is compatible with the action of $\Gal(L/k)$ on $H^1(L,G)$. The image of $x^*f$ in $\prod_{w \mid v} H^1(L_w, G)$ is fixed by this action, so we conclude that $\xi \in H^1(L,G)^{\Gal(L/k)}$. By Lemma~\ref{lem:autaia} we have that $\xi$ is the image of some $\xi' \in H^1(k,G)$. Since the maps are all injective, $x^*f$ must be the image of $\xi'$. This means $x \in X(\A_k)^f$.
\end{proof}

\begin{rem}
	In the proof above, the condition $G(L)=0$ is only used to ensure the existence of a lift of $x^*(f)$ from $\prod_v H^1(k_v, G)$ to $H^1(k, G)$, using that its image in $\prod_v\prod_{w \mid v} H^1(L_w, G)$ comes from an element in $H^1(L, G)$. The conclusion of Theorem~$\ref{thm:Gksep0}$ therefore holds for any (possibly infinite) abelian group scheme $G$ such that Diagram~\eqref{diagram:autaia} is Cartesian for all finite Galois extensions $L/k$. 
\end{rem}

\begin{proof}[Proof of Theorem~\ref{thm:MainThm2}]
	Theorem~\ref{thm:etale} gives the implication~\eqref{cond:1} $\Rightarrow$ ~\eqref{cond:2}  and Theorem~\ref{thm:Gksep0} gives the converse.
\end{proof}

\begin{rem}\label{rem:a'}
	Let $f: Y \to X$ be as in the statement of Theorem~\ref{thm:MainThm2} and assume further that $X$ is proper. Theorem~\ref{thm:etale} shows that Condition (a') of the introduction implies~\eqref{cond:2}. Since~\eqref{cond:1} implies (a'), we have that~\eqref{cond:1}, (a') and~\eqref{cond:2} are all equivalent.
\end{rem}

\section{Subvarieties of abelian varieties}\label{sec:BM}

Throughout this section $X \subset A$ denotes a geometrically irreducible smooth closed subvariety of an abelian variety over a global field $k$.

\begin{defn}\label{def:EtSheaf}
	We say that $X(\A_k)_\bullet^{\Br}$ \emph{defines an \'etale sheaf} if, for all finite Galois extensions $L/k$, we have that the image of $X(\A_k)_\bullet^{\Br}$ in $X(\A_L)_{\bullet}$ is equal to $\left(X(\A_L)_\bullet^{\Br}\right)^{\Gal(L/k)}$, and similarly for the other sets appearing in~\eqref{diagram1}.
\end{defn}

\begin{rem}
When $k$ is totally imaginary or a global function field, the definition is equivalent to the assignment $L \mapsto X(\A_L)^{\Br}_\bullet$ defining a sheaf of sets on $\Spec(k)_{\eT}$ in the usual sense. See the discussion around Condition (a') in the introduction.
\end{rem}

\subsection{The finite abelian descent set}
Following \cite[p. 352]{Stoll} we define $\widehat{\Sel}(A/k) = \varprojlim_n \Sel^n(A/k)$. This fits into the Cassels-Tate dual exact sequence \cite[Proposition 4.3]{PoonenVoloch}, which reads
\[
	0 \to \widehat{\operatorname{Sel}}(A/k) \to A(\A_k)_{\bullet} \stackrel{\phi}\to H^1(k,A^\vee)^*\,.
\]
The map $\phi$ is induced by the sum of the local Tate pairings $\langle \,,\,\rangle_{k_v} : A(k_v) \times H^1(k_v,A^\vee) \to \Q/\Z$. This sequence identifies $\widehat{\Sel}(A/k)$ with a subset of $A(\A_k)_\bullet$.

\begin{thm}\label{thmSelhat}
The sets	$\widehat{\Sel}(A/k)$ and $X(\A_k)_\bullet \cap \widehat{\Sel}(A/k)$ define \'etale sheaves.
\end{thm}

\begin{proof}
Let $L/k$ be any finite Galois extension and let $d=[L:k]$. It suffices to prove the result for $\widehat{\Sel}(A/k)$, since $$\left(X(\A_L)_\bullet \cap \widehat{\Sel}(A/L)\right)^{\Gal(L/k)} = X(\A_L)_\bullet^{\Gal(L/k)} \cap \widehat{\Sel}(A/L)^{\Gal(L/k)}\,,$$ and $X(\A_L)_\bullet^{\Gal(L/k)}$ is the image of $X(\A_k)_\bullet$ in $X(\A_L)_\bullet$.

Let $x_L \in\widehat{\Sel}(A/L)^{\Gal(L/k)} \subseteq A(\A_L)_\bullet^{\Gal(L/k)}$. Since $A(\A_k)_\bullet$ defines an étale sheaf, there exists an element of $A(\A_k)_\bullet$ mapping to $x_L$ under the natural map $A(\A_k)_\bullet \to A(\A_L)_\bullet$, which we will call $x_k$. We first claim that $dx_k \in \widehat{\operatorname{Sel}}(A/k)$. For this we use the Cassels-Tate dual exact sequence. For any place $v$, passing to an extension $L^v/k_v$ of degree $d_v$ multiplies the local Tate pairing by $d_v$, i.e., we have $\langle x_{k_v},\textup{res}_{L^v/k_v}(\alpha) \rangle_{L^v} =  d_v\langle x_{L_v} , \alpha \rangle_{k_v}$. From this we deduce that for any finite extension $K/k$ we have $[K:k]\phi(x_k) = \phi_K(x_K) \circ \res_{K/k} \in H^1(k,A^\vee)^*$, where we write $x_K$ to denote the image of $x_k$ in $A(\A_K)_\bullet$. The assumption that $x \in \widehat{\Sel}(A/L)^{\Gal(L/k)}$ together with exactness of the sequence gives $\phi_L(x_L) = 0$, so $\phi(dx_k) = d\phi(x_k) = \phi_L(x_L)\circ \res_{L/k} = 0$, showing that $dx_k \in \widehat{\Sel}(A/k)$. 

Since $dx_k \in \widehat{\operatorname{Sel}}(A/k) = \varprojlim_{n} \Sel^n(A/k)$, there is a compatible system of geometrically connected torsors $f_n : Y_n \to A$ under $A[n]$ containing lifts $y_n \in Y_n(\A_k)$ of $dx_k$. Here compatible means that for each $m,n$, we have a torsor structure $Y_{mn} \to Y_n$ under $A[m]$ sending $y_{mn}$ to $y_n$. The trivial torsor $[d] : A \to A$ contains a lift of $dx_k$ by hypothesis, so $f_d$ must be a twist of this trivial covering by an element $\xi \in \Sha^1(k,A[d])$. For each $n \ge 1$, let $g_{nd} : Z_{nd} \to A$ be the twist of $f_{nd}$ by the image of $-\xi$ under the map $\Sha^1(k,A[d]) \to \Sha^1(k,A[nd])$ induced by the inclusion $A[d] \hookrightarrow A[nd]$. Then $g_d =  [d] : A \to A$ and $g_{dn}$ factors as $g_{dn} = [d] \circ h_n$ for some torsor $h_n : Z_{nd} \to A$ under $A[n]$. Since $\xi$ is locally trivial and the $f_{nd}$ contain lifts of $dx_k$, we have we have $dx_k^*g_{nd} = dx_k^*f_{nd} = 0$, for all $n$. It follows that the family $h_n$ determines an element in $\widehat{\Sel}(A/k)$, and consequently an adelic point $x'_k \in A(\A_k)_\bullet$. By construction $dx'_k = dx_k$, so $x'_k-x_k \in A[d](\A_k)_\bullet$ is the connected component of adelic point whose nonarchimedean components are contained in a finite subscheme of $A$. 

Consider then the element in $A[d](\A_L)_\bullet$ given by the image of $x'_k-x_k$, which we call $x'_L-x_L$. By assumption $x'_L - x_L$ lies in $\widehat{\operatorname{Sel}}(A/L)$. Using \cite[Theorem 3.11]{Stoll} in the number field case and \cite[Proposition 5.3]{PoonenVoloch} in the function field case (Note that the additional hypothesis on $A$ there can be dropped thanks to work of R\"ossler; see \cite[Proposition 3.1]{CVNonIsotriv}) we obtain $x'_L-x_L \in A[d](L) \subseteq A(L)$. However, $x'_L-x_L$ is also fixed by $\Gal(L/k)$, so $x'_L-x_L \in A(L)^{\Gal(L/k)}$. Since $A(k) \to A(L)^{\Gal(L/k)}$ is always a bijection, there exists a unique element $x'' \in A(k) \subseteq A(\A_k)_\bullet$ which maps to $x'_L-x_L$ under the base change map. Consider then the element $x''-x'_k \in A(\A_k)_\bullet$. Since $x'' \in A(k)$, the element $x''-x'_k$ lies in $\widehat{\Sel(A/k)}$ and maps to $x_L$ under the map $A(\A_k)_\bullet \to A(\A_L)_\bullet$, so the map $\widehat{\Sel(A/k)} \to \widehat{\Sel(A/L)}^{\Gal(L/k)}$ is surjective as required. 
\end{proof}

\begin{proof}[Proof of Theorem~\ref{thm:Main2}]
	As noted on \cite[p. 373]{Stoll} the image of $\widehat{\Sel}(A/k)$ in $A(\A_k)_\bullet$ is equal to $A(\A_k)_\bullet^{\textup{f-ab}}$. With this identification Theorem~\ref{thm:Main2} is just a restatement of Theorem~\ref{thmSelhat}.
\end{proof}

\begin{proof}[Proof of Corollary~\ref{cor:Main2}]
	To prove that $\widehat{\Sel}(A/k)$ equals $A(\A_k)_\bullet^{\textup{f-ab}}$ one uses that the N\'eron-Severi group of $A$ is torsion free, and so the geometrically abelian fundamental group of $A$ is isomorphic to its Tate module. The same argument works to show $X(\A_k)_\bullet^{\textup{f-ab}} = X(\A_k)_\bullet\cap \widehat{\Sel}(A/k)$ when the N\'eron-Severi group of $X$ is torsion free and $X \subset A$ is embedded by an Albanese map (See \cite[Remark 6.5]{Stoll}). 
\end{proof}

\subsection{Sheafiness of the other terms in~\eqref{diagram1}}

\begin{lemma}\label{lem:1}
The sets $\overline{A(k)}$ and $X(\A_k)_\bullet \cap \overline{A(k)}$ define \'etale sheaves.
\end{lemma}

\begin{proof}
The topological closure of $A(L)$ in $A(\A_L)_\bullet$ and the profinite completion $\widehat{A(L)}$ are isomorphic as $\Gal(L/k)$-modules by \cite[Theorem E]{PoonenVoloch}. Since $A(L)$ is finitely generated we also have  an isomorphism of $\Gal(L/k)$-modules $\widehat{A(L)} \simeq A(L) \otimes \widehat{\Z}$. Then $(A(L) \otimes \widehat{\Z})^{\Gal(L/k)} = A(L)^{\Gal(L/k)}\otimes \widehat{\Z} = A(k) \otimes \widehat{\Z}$, the latter being identified with $\overline{A(k)}$ by \cite[Theorem E]{PoonenVoloch}. Thus $\overline{A(k)}$ defines an \'etale sheaf. The statement about $X(\A_k)_\bullet \cap \overline{A(k)}$ follows instantly since 
$$
\left(X(\A_L)_\bullet \cap \overline{A(L)}\right)^{\Gal(L/k)} = X(\A_L)_\bullet^{\Gal(L/k)} \cap \overline{A(L)}^{\Gal(L/k)}\,,$$ and $X(\A_L)_\bullet^{\Gal(L/k)}$ is equal to the image of $X(\A_k)_\bullet$ in $X(\A_L)_\bullet$.

\end{proof}

\begin{lemma}\label{lem:2}
If $k$ is a number field or the maximal divisible subgroup of $\Sha(A/k)$ is trivial, then $A(\A_k)_\bullet^{\Br}$ defines an \'etale sheaf.
\end{lemma}

\begin{proof}
In the number field case, \cite[Theorem 1]{CTrBM} implies $A(\A_k)_\bullet^{\Br} = A(\A_k)_\bullet^\textup{f-ab}$, so this follows from Theorem~\ref{thmSelhat}. If the divisible subgroup of $\Sha(A/k)$ is trivial then $\overline{A(k)} = A(\A_k)_\bullet^{\Br} = A(\A_k)_\bullet^{\textup{f-ab}}$ (See \cite[Remark 4.5]{PoonenVoloch}), so this follows from either Lemma~\ref{lem:1} or Theorem~\ref{thmSelhat}.
\end{proof}

\begin{rem}\label{rem:2}
	When $A/k$ is an isotrivial abelian variety, we have that $\Sha(A/k)$ is finite. For $A/k$ constant, this is shown in \cite{Milne68}, and as mentioned in \cite[Remark 6.27]{MilneADT} it is an easy extension to the isotrivial case.
\end{rem}

\begin{lemma}\label{lem:3}
Let $X \subset A$ be a curve embedded in its Jacobian by an Albanese map (i.e., a map sending a point $P$ to the class of $P-D$ for a fixed $k$-rational divisor $D \in \operatorname{Div}(X)$ of degree $1$). Suppose that $X$ has genus $1$ or is nonisotrivial. Then $\overline{X(k)}$ defines an \'etale sheaf.
\end{lemma}

\begin{proof}
If $X$ has genus $1$, then $X = A$ and this follows from Lemma~\ref{lem:1}. So suppose $X$ has genus $\ge 2$. If $k$ is a number field or $X$ is a nonisotrivial, then $X(k)$ is finite. Then $\overline{X(k)}=X(k)$ which clearly defines an \'etale sheaf.
\end{proof}

Lemma~\ref{lem:CbarIsotrivial} of the following section shows that Lemma~\ref{lem:3} also holds for isotrivial curves. We will postpone proving this as it requires some setup particular to isotrivial varieties. Admitting this for now we have the following.

\begin{thm}\label{thm:Curves}
	Suppose $X \subset A$ is a curve embedded in its Jacobian by an Albanese map (i.e., a map sending a point $P$ to the class of $P-D$ for a fixed $k$-rational divisor $D \in \operatorname{Div}(X)$ of degree $1$). Then all of the sets in~\eqref{diagram1} define \'etale sheaves.
\end{thm}

\begin{proof}
	First note that $X(\A_k)_\bullet \cap A(\A_k)_\bullet^{\Br} = X(\A_k)_\bullet^{\Br} = X(\A_k)_\bullet^{\textup{f-ab}} = X(\A_k)_\bullet \cap A(\A_k)_\bullet^{\textup{f-ab}}$ (See \cite[Corollary 7.3]{Stoll}). So all of these define \'etale sheaves by Theorem~\ref{thmSelhat}. We have that $X(\A_k)_\bullet \cap \overline{A(k)}$ defines an \'etale sheaf by Lemma~\ref{lem:1}. If $X$ is nonisotrivial or has genus $1$, then $\overline{X(k)}$ defines an \'etale sheaf by Lemma~\ref{lem:3}. If $X$ is isotrivial and of genus $\ge 2$, then $\overline{X(k)}$ defines an \'etale sheaf by Lemma~\ref{lem:CbarIsotrivial}. 
\end{proof}

\begin{rem}\label{rem:3}
	Suppose $C$ is a curve of genus $\ge 1$ over $k$. Then $C$ embeds canonically into $\Pic^1_C$ which is a torsor under the Jacobian $A = \Pic^0_C$. If $C(\A_k) \ne \emptyset$, then $\Pic^1_C$ represents a class in $\Sha(A/k)$ and $C$ has a $k$-rational divisor of degree $1$ if and only if this class in $\Sha(A/k)$ is trivial. If $C(\A_k)_\bullet^{\Br} \ne \emptyset$, then the class of $\Pic^1_C$ is a divisible element in $\Sha(A/k)$ by \cite[Proposition 3.3.5 and Theorem 6.1.2]{SkorobogatovBook}. Thus, for any curve with $\Sha(A/k)$ finite (e.g., any isotrivial curve as noted in Remark~\ref{rem:2}) we either have an Albanese embedding of $C$ into its Jacobian (and so Theorem~\ref{thm:Curves} applies) or we have $\overline{C(k)} = C(\A_k)_\bullet^{\Br} = \emptyset$. \end{rem}

\subsection{Counterexamples among general varieties}

We now show that, in general, none of the sets in the top row of~\eqref{diagram1} define \'etale sheaves, if we consider varieties that do not admit an embedding into an abelian variety.

\begin{propn}\label{BrNotSheaf}
Let $Y/k$ be a smooth projective variety over a number field such that $\Pic(\Ybar)$ is torsion free, $\Br(\Ybar)$ is finite, and $\Br(Y) \to \Br(\Ybar)^{\Gal(k)}$ is surjective.
\begin{enumerate}
	\item\label{it:1} If $Y(\A_k)_\bullet \ne Y(\A_k)_\bullet^{\Br}$, then there exists a finite Galois extension $L/k$ such that the image of $Y(\A_k)_\bullet^{\Br}$ in $Y(\A_L)_\bullet$ is a proper subset of $\left(Y(\A_L)_\bullet^{\Br}\right)^{\Gal(L/k)}$.
	\item\label{it:2} If $Y(k) = \emptyset$, $Y(\A_k)_\bullet \ne \emptyset$ and $\overline{Y(K)} = Y(\mathbb{A}_K)_\bullet^{\Br}$ for all finite Galois extensions $K/k$, then there exists a finite Galois extension $L/k$ such that $\overline{Y(L)}^{\Gal(L/k)} \ne \emptyset = \overline{Y(k)}$.
\end{enumerate}
\end{propn}

\begin{proof}
	The assumptions in the first sentence imply that $\Br(Y)/\Br_0(Y)$ is finite, say of order $d$. By~\cite[Lemma 3.3]{CreutzViray} there are infinitely many Galois extensions of degree divisible by $d$ such that the restriction map $\Res_{L/k} : \Br(Y)/\Br_0(Y) \to \Br(Y_L)/\Br_0(Y_L)$ is surjective. For any such extension, \cite[Lemma 3.1(2)]{CreutzViray} gives that $Y(\A_k)_\bullet \subset Y(\A_L)_\bullet^{\Br}$, so every element $Y(\A_k)_\bullet$ lies in the $\Gal(L/k)$-invariant subset of $Y(\A_L)_\bullet^{\Br}$. The claims in \eqref{it:1} and \eqref{it:2} follow immediately.
\end{proof}

\begin{rem}\label{rem:BrNotSheaf}
Examples satisfying the conditions of Proposition~\ref{BrNotSheaf}\eqref{it:1} can be found among del Pezzo surfaces and Ch{\^a}telet surfaces. Ch{\^a}telet surfaces that are counterexamples to the Hasse principle satisfy the conditions in~Proposition~\ref{BrNotSheaf}\eqref{it:2} by \cite[Theorem B]{CTS}. 
 \end{rem}

 \begin{propn}\label{cor:enriques}
Let $Y/k$ be a smooth projective Enriques surface over a global field $k$ of characteristic not equal to $2$. Then the assignment $L \mapsto Y(\A_L)_\bullet^{\textup{f-ab}}$ does not define an \'etale sheaf.
\end{propn}
\begin{proof}
Let $f: Z \to Y$ be a $K3$ cover of the Enriques surface, so that $Z \to Y$ is a torsor under $\Z/2\Z$. Since $Z$ is a $K3$ surface, it is \'etale simply connected, and so $Y(\A_L)_\bullet^f = Y(\A_L)_\bullet^{\eT} = Y(\A_L)_\bullet^{\textup{f-ab}}$. The result follows by Theorem~$\ref{thm:etale}$.
\end{proof}

\section{Isotrivial curves}
In this section we use results of the previous sections to generalise the main results of \cite{CVBM} to the case of isotrivial curves.

Fix a finite field $\F$ of characteristic $p$. Let $D$ be a smooth projective geometrically connected curve over $\F$, and let $k = \F(D)$ denote the function field of $D$. Throughout this section $C$ will denote a smooth projective and geometrically irreducible curve over $k$ of genus $g = g(C) \ge 2$ which we assume to be isotrivial. We also assume that $C$ has a $k$-rational divisor $z$ of degree $1$ and use this to define an embedding of $C$ into its Jacobian $J = \Jac(C)$ by the rule $x \mapsto [x-z]$. This assumption is justified by Remark~\ref{rem:3}.

There exists a finite extension $L/k$ with corresponding extension of constant fields $\F_L/\F$ and a curve $C_0/\F_L$ such that $C\times_k L \simeq C_0 \times_{\F_L} L$. For any such $L/k$ we also have that $J \times_k L \simeq J_0 \times_{\F_L} L$, where $J_0 = \Jac(C_0)$ is an abelian variety defined over $\F_L$. One can take $L/k$ to be separable because the moduli space of curves with sufficiently large level structure is a fine moduli space, and replacing $L$ by its Galois closure we obtain a Galois extension of $k$ trivialising $C$.
 
\begin{rem}
There are isotrivial varieties that are not separably isotrivial, meaning that they only become constant after a non-separable field extension, e.g. singular genus-changing curves in the sense of \cite{T}. It is possible that there exist smooth examples in higher dimension but we do not know of any.
\end{rem}

The set of places of $k$ is in bijection with the set $D^1$ of closed points of $D$. For $v \in D^1$ we denote the residue field of the completion $k_v$ by $\F_v$.  Note that since the valued field $k_v$ has equicharacteristic $p$, $\F_v \subseteq k_v$. We define $\mathbb{A}_{k, \F} := \prod_{v \in D^1} \F_v$, which is an $\F$-subalgebra of the usual adele ring $\A_k$. If $L/k$ is a finite extension with constant field extension $\F_L/\F$, then there exists a smooth projective curve $D'/\F_L$ with $L = \F_L(D')$, so we may define $\A_{L,\F_L}$ similarly.

\subsection{Locally constant adelic points}

We recall the definition of locally constant adelic points as in \cite[Section 2.2]{CVBM} (where they were called reduced adelic points) before extending this definition to isotrivial varieties. Suppose that $X/k$ is a constant variety so that $X = X_0 \times_{\F} k$. Since $D/\F$ is geometrically connected, we have natural equalities of sets
$$
\Hom_{/k}(\Spec(\mathbb{A}_k), X) = \Hom_{/\F}(\Spec(\mathbb{A}_k), X) = \Hom_{/\F}(\Spec(\mathbb{A}_k), X_0)\,.
$$
\begin{defn}\label{locconstadpts}
For a constant variety $X = X_0 \times_{\F} k$ over $k$ we define the \emph{locally constant adelic points} to be the set 
$$
X(\mathbb{A}_{k, \F}) := \Hom_{/\F}( \Spec(\mathbb{A}_{k, \F}), X_0)\,.
$$
\end{defn}

Note that $\mathbb{A}_{k, \F}$ is an $\F$-subalgebra of $\mathbb{A}_k$, so we have an inclusion
$$
X(\mathbb{A}_{k, \F}) \subseteq \Hom_{/\F}(\Spec(\mathbb{A}_k), X_0) = \Hom_{/k}(\Spec(\mathbb{A}_k), X) = X(\mathbb{A}_k)\,.
$$
Concretely, $X(\mathbb{A}_{k,\F}) = \prod_{v \in D^1} X_0(\F_v)$.

\begin{defn}
Suppose $X/k$ is an isotrivial variety and $L/k$ is a Galois extension with corresponding residue extension $\F_L/\F$ such that $X_L/L$ is a constant variety. Let $X_0/\F_L$ be the corresponding constant variety and let $\phi: X_L \to X_0 \times_{\F_L} L$ be an isomorphism. Define the \emph{locally constant adelic points} of $X/k$ to be the set
\[
	X(\A_{k,\F}) := X(\A_L)^{\Gal(L/k)} \cap \phi^{-1}((X_0 \times_{\F_L} L)(\A_{L,\F_L})) \subset X(\A_L)^{\Gal(L/k)} = X(\A_k)\,.
\]
\end{defn}
 The following two lemmas show that this definition does not depend on our choice of trivialising extension $L$ and isomorphism $\phi$.
 
\begin{lemma}\label{lem:welldefinediso}
Let $X/k$ be an isotrivial variety, and let $\phi, \phi': X_L \to X_0 \times_{\F} L$ be two isomorphisms of $L$ varieties. Then $\phi^{-1}((X_0\times_{\F} L)(\A_{L,\F})) = \phi'^{-1}((X_0\times_{\F} L)(\A_{L,\F}))$. In particular, the set $X(\A_{k, \F})$ does not depend on our choice of trivialising isomorphism $\phi$. 
\end{lemma}
\begin{proof}
Note that $\phi' \circ \phi^{-1}$ gives us an automorphism of $X_0 \times_{\F} L$ as a variety over $L$, so the result is equivalent to saying that the set $(X_0\times_{\F} L)(\A_{L,\F})$ is stable under the action of $\mathrm{Aut}_{/L}(X_0 \times_{\F}L)$.  This is immediate by the definition of locally constant adelic points for constant varieties.
\end{proof}
\begin{lemma}\label{lem:welldefinedext}
Let $X/k$ be an isotrivial variety. Then the definition of $X(\A_{k, \F})$ does not depend on our choice of trivialising extension.
\end{lemma}
\begin{proof}
Let $L, L'$ be two trivialising extensions for $X/k$, and consider the extension $K := LL'$. Since $X_L$ is constant, we have that $X_L( \A_{L, \F_L}) = X_K(\A_{K, \F_K}) \cap X_L(\A_L)$, and similarly for $X_{L'}(\A_{L', \F_{L'}})$. Therefore 
$$
X_L( \A_{L, \F_L}) \cap X(\A_k) = X_K(\A_{K, \F_K}) \cap X(\A_k) = X_{L'}( \A_{L', \F_{L'}}) \cap X(\A_k)
$$
as required.
\end{proof}

\begin{rem}
Locally constant adelic points are defined precisely so that they define an étale sheaf in the sense of Definition~$\ref{def:EtSheaf}$, and their use is justified by the property they are shown to satisfy in Theorem $\ref{thm:Frobdescent}$.
\end{rem}
\subsection{The Frobenius map on isotrivial varieties}

Let $X/k$ be an isotrivial variety and $L/k$ a Galois extension such that there is an isomorphism $\phi: X_L \simeq X_0 \times_{\F_L} L$ for some $X_0/\F_L$, where $\F_L/\F$ is the residue extension corresponding to $L/k$. So $X$ is a twist of $X_0$ and is thus classified by some class in $H^1(\Gal(L/k),\Aut(X_0))$. This class is represented by some cocycle $\xi:\Gal(L/k)\to \Aut(X_0)$ and the image of $\xi$ is some finite subset of $\Aut(X_0)$. We can extend $\F_L$ (and subsequently $L$) and assume that these finitely many automorphisms of $X_0$ are defined over $\F_L$. We say that $\phi$ has trivialised Galois action in this case.

\begin{lemma}\label{lem:Frob}
Suppose $X/k$ is an isotrivial variety and $L/k$ is a Galois extension such that there is an isomorphism $\phi: X_L \simeq X_0 \times_{\F_L} L$ with trivialised Galois action in the above sense for some $X_0/\F_L$ where $\F_L/\F$ is the residue extension corresponding to $L/k$. Let $m:=[\F_L:\F]$. The relative Frobenius morphism $F_{X_0/\F_L} : X_0 \to X_0$ induces a morphism $F_{X_L/L}: X_L \to X_L$ that descends to $k$ and yields a morphism $X \to X$, which we call $F^m_{X/k}$.
\end{lemma}

\begin{example}
	Consider an isotrivial curve $C : ty^2 = f(x)$ over $\F_p(t)$ with $f(x) \in \F_p[x]$ for an odd prime $p$. The relative Frobenius for $X/k$ is the morphism $F : C \to C^{(p)}$ given on coordinates by raising to the $p$-th power, where $C^{(p)}/\F$ is the curve given by $t^py^2 = f(x)$. Since $p$ is odd, we have $C^{(p)} \simeq C$ by the map $(x,y) \mapsto (x,t^{\frac{p-1}{2}}y)$. The Galois extension $L = \F_p(t^{1/2})/\F_p(t)$ trivialises $C$. The morphism $F_{C/k}^1$ constructed in the lemma is the composition of $F$ and the isomorphism $C^{(p)}\simeq C$.
\end{example}

\begin{proof}[Proof of Lemma~\ref{lem:Frob}]

By assumption, the morphism $\phi$ has trivialised Galois action, so the 
values of the cocycle $\xi(\sigma) = \phi^{\sigma}\circ\phi^{-1} \in \Aut(X_0)$
are defined over $\F_L$ thus commute with $F_{X_0/\F_L} : X_0 \to X_0$.
Therefore $F_{X_L/L} := \phi^{-1}\circ F_{X_0/\F_L} \circ \phi$ satisfies
$$F_{X_L/L}^{\sigma} = (\phi^{-1})^{\sigma}\circ F_{X_0/\F_L} \circ \phi^{\sigma} = \phi^{-1}\circ \xi(\sigma)^{-1} \circ F_{X_0/\F_L} \circ \xi(\sigma) \circ \phi = F_{X_L/L}$$
so it descends to $k$, as claimed.

\end{proof}

\subsection{Sheafiness of $\overline{C(L)}$ in the isotrivial case}

Recall that $\overline{C(k)}$ denotes the topological closure of $C(k)$ in $C(\A_k)$. We now prove that $\overline{C(k)}$ defines an \'etale sheaf in the isotrivial case.

\begin{lemma}\label{lem:CbarIsotrivial}
	Let $C/k$ be a smooth projective geometrically irreducible and isotrivial curve of genus at least $2$ over $k$. Let $L/k$ be a finite Galois extension. Then $\overline{C(L)}^{\Gal(L/k)} = \overline{C(k)}$.
\end{lemma}

\begin{proof}
	We claim that we can assume without loss of generality that $L/k$ trivialises $C/k$,  i.e., $L/k$ is such that $C\times_k L \simeq C_0 \times_{\F_L} L$ for some curve $C_0/\F_L$. Let $K/k$ be a Galois extension such that $L \subseteq K$ and $K/k$ trivialises $C/k$. Then $\overline{C(K)}^{\Gal(K/k)} = \left(\overline{C(K)}^{(\Gal(K/L))}\right)^{\Gal(L/k)}$, so showing the result for Galois extensions that trivialise $C/k$ implies the result for all Galois extensions.

	Let $D_L$ be the smooth projective curve with $L = \F_L(D_L)$. Let $F = F_{C_0/\F_L}$ be the relative Frobenius morphism. By a theorem of de Franchis (\cite[pg 223-224]{Lang} or \cite[pg 3]{CVetale}) the set $C(L)/F$ is finite. Choose a set of distinct elements $\phi_1,\dots,\phi_m \in C(L) = \operatorname{Mor}_{\F_L}(D_L,C_0)$ representing all morphisms from $D_L$ to $C_0$ up to Frobenius twisting. Concretely this means that for any $\phi \in C(L)$, there exists integers $a,b \ge 0$ and an $i \in \{1,\dots,m\}$ such that $F^a\phi = F^b\phi_i$. We can assume that the $\phi_i$ are distinct, separable and the set $\{\phi_1,\dots,\phi_m \}$ is closed under the action of $\Gal(L/k)$. The map $C(L)/F \to \operatorname{Map}(D_L^1,C_0^1)$ is injective by \cite[Proposition 2.3]{Stix02} and there are only finitely many $\phi_i$, so we can find a finite set $S$ of places $v \in D_L^1$ such that the elements $r_S(\phi_i) := (r_v(\phi_i))_{v \in S} \in \prod_{v \in S}C_0(\F_v)$ are distinct for $i \ne j$, where $r_v : C_0(L_v) \to C_0(\F_v)$ is the reduction map. Enlarging $S$ if needed, we can assume that $r_S(F^a\phi_i) \ne r_S(F^b\phi_j)$ for any $a,b \ge 0$ and $i\ne j$. Note that $r_v(F^a\phi_i) = F^a(r_v(\phi_i))$, where on the right $F$ acts via $\Gal(\F_L)$ on $C_0(\F_v)$.
	
	Let $\rho_v$ be metrics inducing the $v$-adic topology on $C(L_v)$ and let $\rho$ be a product metric of these inducing the adelic topology on $C(\A_L)$. The reduction maps give rise to a continuous retraction $r = (r_v) : C(\A_L) \to C(\A_{L,\F_L})$. By the discussion above, the set of distances $\{ \rho(r(F^a\phi_i),r(F^b\phi_j)) \;|\; a,b \ge 0 \text{ and } i \ne j \}$ is bounded away from $0$.

	Suppose $P_n \in C(L)$ converge to $P \in \overline{C(L)} \cap C(\A_k)$. For any $\sigma \in \Gal(L/k)$ the sequence $\sigma(P_n)$ converges to $\sigma(P) = P$ in $C(\A_L)$, since the induced map $\sigma : C(\A_L) \to C(\A_L)$ is continuous. Thus the sequence of real numbers $\rho(P_n,\sigma(P_n))$ converges to $0$. For all $n \ge 1$ we have integers $a_n,b_n \ge 0$ and $i_n \in \{ 1,\dots, m\}$ such that $F^{a_n}P_n = F^{b_n}\phi_{i_n}$. Since $F : C(\A_L) \to C(\A_L)$ is continuous and commutes with the action of $\sigma$, we find that the sequence $\rho(F^{b_n}\phi_{i_n},F^{b_n}\sigma(\phi_{i_n}))$ converges to $0$. Since the reduction map is continuous this implies that $\rho(r(F^{b_n}\phi_{i_n}),r(F^{b_n}\sigma(\phi_{i_n})))$ converge to $0$. These distances are bounded away from zero when $\phi_{i_n} \ne \sigma(\phi_{i_n})$. So for all large enough $n$ we must have $\sigma(\phi_{i_n}) = \phi_{i_n}$. Since the action of $\sigma$ commutes with $F$ we find that, for all large enough $n$, we have $F^{a_n} \sigma(P_n) = F^{a_n} P_n$ and so $P_n = \sigma(P_n)$ as well. Thus the $P_n$ are eventually in $C(L)^{\Gal(L/k)} = C(k)$. Hence $P \in \overline{C(k)}$.	
\end{proof}

\subsection{Frobenius descent}

Let $L/k$ be the minimal Galois extension trivialising the isotrivial curve $C/k$. Then $L/k$ also trivialises $J = \Jac(C)$ and Lemma~\ref{lem:Frob} gives an isogeny $F^m_{J/k} : J \to J$ where $m = [\F_L:\F]$. 

\begin{lemma}\label{lem:lcFrob}
The locally constant adelic points of the Jacobian satisfy $J(\A_{k,\F}) \subset F^m_{J/k}(J(\A_k))$.
\end{lemma}

\begin{proof}
	
	This is clear for constant varieties and Frobenius commutes with the Galois action.
\end{proof}

 For any $n \ge 1$, the $n$-fold composition of $F^m_{J/k}$ is an isogeny $\phi_n = (F^m_{J/k})^{\circ n} : J \to J$ whose kernel is a finite connected abelian group scheme. The pullback of $\phi_n$ along the embedding $C \to J$ is a torsor $F^n : Y \to C$ under $\ker(\phi_n)$. We define $C(\A_k)^{F^\infty} = \bigcap_{n \ge 1} C(\A_k)^{F^n}$.
 
 The following generalises~\cite[Theorem 1.2]{CVBM} to the case of isotrivial curves.

\begin{thm}\label{thm:Frobdescent}
Only global and locally constant adelic points survive infinite Frobenius descent, i.e.,
$
C(\mathbb{A}_k)^{F^{\infty}} = C(k) \cup C(\mathbb{A}_{k,\F}).
$
\end{thm}
\begin{proof}

For a finite separable extension $L/k$ with corresponding residue extension $\F_L/\F$ define $\mathcal{F}(L) := C(L) \cup C(\mathbb{A}_{L, \F_L})$ and $\mathcal{G}(L) := C(\mathbb{A}_L)^{\mathrm{ker}(F^\infty)}$. Then $\mathcal{F}$ and $\mathcal{G}$ are sheaves of sets on $\Spec(k)_\eT$, the latter by Theorem~\ref{thm:Gksep0}. By Lemma~\ref{lem:lcFrob}, all reduced adelic points survive Frobenius descent, so $\mathcal{F}$ is a subsheaf of $\mathcal{G}$. For any  $L/k$ trivialising $C$, we may apply \cite[Theorem 1.2]{CVBM} in order to obtain $\mathcal{F}(L) = \mathcal{G}(L)$. This shows that the inclusion $\mathcal{F} \subset \mathcal{G}$ is an isomorphism \'etale locally. Therefore $\mathcal{F} = \mathcal{G}$ by \cite[\S II Proposition 1.1]{Ha} and, in particular, we have have $\mathcal{F}(k) = \mathcal{G}(k)$.
 \end{proof}

\subsection{The Mordell-Weil Sieve}
\begin{defn}
We define the \emph{Mordell--Weil Sieve set} for an isotrivial curve embedded in its Jacobian to be the set
\[
	C^{\textup{MW-sieve}} := C(\A_{k,\F}) \cap \overline{J(k)} \,,
\]
with this intersection taking place in $J(\A_k)$.
\end{defn}
 When $C/k$ is a constant curve this agrees with the definition of $C^{\textup{MW-sieve}}$  in \cite{CVBM} so the following result generalises \cite[Theorem 1.1]{CVBM} to isotrivial curves.

\begin{thm}\label{thm:MW}
For $C/k$ a smooth, projective, isotrivial curve embedded in its Jacobian, we have an equality $C(\A_k)^{\Br} = C(k) \cup C^{\textup{MW-sieve}}$.
\end{thm}

\begin{proof}
	The assignments $L \mapsto \mathcal{G}(L) := C(\A_L)^{\Br}$ and $L \mapsto \mathcal{F}(L):= C(L) \cup \left( C(\A_{L,\F'}) \cap \overline{J(L)}\right)$ define \'etale sheaves by Theorem~\ref{thm:Curves} and Lemma~\ref{lem:1}. Moreover, $\mathcal{F}(L) \subset \mathcal{G}(L)$ for all finite separable extensions $L/k$ since $C(\A_L)^{\Br} \subset J(\A_L)^{\Br} = \overline{J(L)}$, where the final equality is given in the proof of Lemma~\ref{lem:2}, noting that $\Sha(A/L)$ is finite by Remark~\ref{rem:2}. For any $L/k$ trivialising $C$ we have $\mathcal{F}(L) = \mathcal{G}(L)$ by \cite[Theorem 1.1]{CVBM}, so we proceed as in the proof of Theorem~\ref{thm:Frobdescent} to obtain $\mathcal{G}(k) =\mathcal{F}(k)$ as required.
\end{proof}

\begin{rem}
If $C$ is an isotrivial curve that does not admit an embedding in its Jacobian, Remark~$\ref{rem:3}$ shows that $C(\A_k)^{\Br}=\emptyset$. For $C$ nonisotrivial over a global function field and of genus $\geq 2$, \cite[Theorem 1.1]{CVNonIsotriv} gives that $C(\A_k)^{\Br}=C(k)$. The above result is a step towards a characterisation of $C(\A_k)^{\Br}$ for curves in the remaining cases over global function fields.
\end{rem}



\begin{thebibliography}{99}
\bibitem{CL} Yang Cao and Yongqi Liang, {\it Etale Brauer-Manin obstruction for Weil restrictions}, Advances in Mathematics, {\bf 410} (2022).	
	
\bibitem{CTS} J.~L.~Colliot-Th{\'e}l{\`e}ne and J.~J.~Sansuc, {\it Intersections of two quadrics and Ch{\^a}telet surfaces I.}, Journal f{\"u}r die reine und angewandte Mathematik, \textbf{373}, (1987), 37--107.

\bibitem{CTrBM} B.~Creutz, {\it There are no transcendental Brauer--Manin obstructions on abelian varieties}, International Mathematics Research Notices \textbf{2020}, (2020), 2684--2697.

\bibitem{CreutzViray} B.~Creutz and B.~Viray, {\it Quadratic points on intersections of two quadrics}, Algebra and Number Theory \textbf{17}, 2023, 1411-1452.

\bibitem{CVBM} B.~Creutz and J. F.~Voloch, {\it The Brauer-Manin obstruction for constant curves over global function fields}, Annales de l'Institute Fourier \textbf{72}, (2022) 43--58.

\bibitem{CVetale} B.~Creutz and J. F.~Voloch, {\it Etale descent obstruction and anabelian geometry of curves over finite fields}, EPIGA, to appear.

\bibitem{CVNonIsotriv}B.~Creutz and J. F.~Voloch, {\it The Brauer-Manin obstruction for nonisotrivial curves over global function fields}, preprint, (2023). 

\bibitem{CVV} B.~Creutz, B.~Viray and J. F.~Voloch, {\it The $d$-primary Brauer-Manin obstruction for curves}, Res. Number Theory \textbf{4}, 26, (2018).

\bibitem{GA-T} C.~Gonz{\'a}lez-Avil{\'e}s and Ki-Seng Tan, {\it On the Hasse principle for finite group schemes over global function fields}, Math. Res. Lett. {\bf 19} (2012), no.2, 453--460.

\bibitem{Ha} R.~Hartshorne, {\it Algebraic Geometry}, Grad. Texts in Math., {\bf 52}, Springer-Verlag, (1977).

\bibitem{Lang} S. Lang, {\it Fundamentals of {D}iophantine geometry},{Springer-Verlag, New York},(1983).

\bibitem{Krull} W.~Krull,{\it \"{U}ber einen {E}xistenzsatz der {B}ewertungstheorie}, Abh. Math. Sem. Univ. Hamburg {\bf 23} (1959) 29--35.

\bibitem{CoN} J.~Neukirch, A.~Schmidt, and K.~Wingberg, {\it Cohomology of Number Fields}, Grundlehren der mathematischen Wissenschaften {\bf 323} (second edition), Springer-Verlag (2008).

\bibitem{LiangHu} Guang Hu and Yongqi Liang, {\it Arithmetic of Ch\^atelet surface bundles revisited}, Forum Mathematicum, {\bf 35} no. 4, (2023), 883--900.

\bibitem{Milne68} J.~S.~Milne, {\it The Tate--Shaferevich group of a constant abelian variety}, Invent. Math. \textbf{6}, (1968), 91--105.

\bibitem{MilneADT} J.~S.~Milne {\it Arithmetic Duality Theorems}, Second Edition, BookSurge LLC, (2006).

\bibitem{MilneAG} J.~S.~Milne {\it Algebraic groups}, Cambridge University Press, (2017).

\bibitem{PoonenHeuristic} B.~Poonen, {\it Heuristics for the Brauer-Manin obstruction for curves}, Exp.Math. {\bf 15(4)}, (2006),415--420.

\bibitem{PoonenBook} B.~Poonen, {\it Rational points on varieties}, Graduate Studies in Mathematics \textbf{186}, American Mathematical Society, Providence, RI, (2017).

\bibitem{PoonenVoloch} B.~Poonen and J. F.~Voloch, {\it The Brauer-Manin obstruction for subvarieties of abelian varieties over function fields}, Annals of Mathematics \textbf{171}, 1, (2010) 511--532. 

\bibitem{RV}C.~Rivera and B.~ Viray, {\it Persistence of the Brauer-Manin obstruction for cubic surfaces}, Math. Research Letters \textbf{29}, (2022), 1881--1889.

\bibitem{Scharaschkin} V.~Scharaschkin {\it Local-global problems and the Brauer-Manin obstruction} Thesis (Ph.D.) - University
of Michigan, ProQuest LLC, Ann Arbor, MI (1999).

\bibitem{Shatz} Shatz, Stephen S {\it Cohomology of artinian group schemes over local fields}, Ann. of Math. (2) {\bf 79} (1964), 411–449.

\bibitem{SkorobogatovBook}A.~Skorobogatov, {\it Torsors and Rational Points}, Cambridge Tracts in Mathematics \textbf{144}, (2001).

\bibitem{StacksProject} {\it The Stacks Project}, 2024.

\bibitem{Stix02} J.~Stix {\it Affine anabelian curves in positive characteristic}, Compositio Math. \textbf{134} (2002), 75--85. 

\bibitem{Stoll} M.~Stoll {\it Finite descent obstructions and rational points on curves}, Algebra Number Theory. \textbf{1:4} (2007) 349-391. 

\bibitem{T} J.~Tate, {\it Genus change in inseparable extensions of function fields}, Proc. Amer. Math. Soc. \textbf{3}, (1952) 400--406.

\bibitem{Wu} H.~Wu, {\it Arithmetic of Ch{\^a}telet surfaces under extensions of base fields}, The Ramanujan Journal, {\bf 62}, (2023) , 997--1010.


\end{thebibliography}
\end{document}